\documentclass{amsart}

\usepackage{graphicx}
\usepackage{amssymb}
\newcommand{\Z}{\mathbb Z}

\newtheorem{theorem}{Theorem}[section]
\newtheorem{proposition}[theorem]{Proposition}
\newtheorem{corollary}[theorem]{Corollary}

\theoremstyle{definition}

\newtheorem{example}[theorem]{Example}

\begin{document}

\title{Some minimal elements for a partial order of prime knots}

\author{Teruaki Kitano and Masaaki Suzuki}


\address{Department of Information Systems Science, 
Faculty of Engineering, 
Soka University, 
Tangi-cho 1-236, 
Hachioji, Tokyo 192-8577, Japan}

\email{kitano@soka.ac.jp}

\address{Department of Frontier Media Science, 
School of Interdisciplinary Mathematical Sciences, 
Meiji University, 
4-21-1 Nakano, Nakano-ku, Tokyo, 164-8525, Japan}

\email{macky@fms.meiji.ac.jp}

\maketitle

\begin{abstract}
A partial order on the set of prime knots can be defined 
by the existence of an epimorphism between knot groups. 
We prove that 
all the prime knots with up to $6$ crossings 
are minimal. 
We also show that each fibered knot with the irreducible Alexander polynomial is minimal. 
\end{abstract}
\section{Introduction}

Let $K$ be a non-trivial prime knot in $S^3$ and $G(K)$ the knot group of $K$, 
which is the fundamental group of $S^3\setminus K$. 
For two knots $K, K'$, 
we write $K \geq K'$ if there exists an epimorphism $\varphi :G(K)\to G(K')$ 
which preserves meridians. 
This relation $\geq$ gives a partial order on the set of prime knots. 
In \cite{KS05, KS11, HKMS11, HKMS12}, this partial order is determined 
for the set of the prime knots with up to $11$ crossings. 
It showed that the knots $3_1,4_1,5_1,5_2,6_1,6_2, 6_3$ are minimal in this set. 
In this paper, we prove that these knots are minimal in the set of all prime knots. 
Here we call a knot $K$ to be minimal 
if $G(K)$ does not surject onto the knot group of any non-trivial knot. 
Note that a knot group always admits an epimorphism onto the knot group of the trivial knot. 

\begin{theorem}
The knots $3_1,4_1,5_1,5_2,6_1,6_2, 6_3$ are minimal in the set of all prime knots. 
\end{theorem}

The same statement holds for another partial order 
derived from an epimorphism between knot groups (not necessary to preserve meridians). 

\section{Proof of results}

First, we review the following fact on a fibered knot. 

\begin{proposition}[\cite {KS08}]\label{fiber_genus}
If $K \geq K'$ and $K$ is fibered, 
then $K'$ is also fibered and $g(K)$ is greater than or equal to $g(K')$, 
where $g(K)$ is the genus of $K$. 
\end{proposition}

\begin{proof}
An epimorphism between knot groups induces 
an epimorphism between their commutator groups. 
Since $K$ is fibered, 
the commutator subgroup $[G(K),G(K)]$ is the free group of rank $g(K)$. 
This implies $[G(K'),G(K')]$ is also finitely generated. 
By using the fibration theorem by Stallings \cite{S}, 
we see that 
$[G(K'),G(K')]$ is isomorphic to the fundamental group of a compact surface $S$
and that $K'$ is a fibered knot with the fiber $S$. 
Further we have $g(K)\geq g(K')$. 
\end{proof}

\begin{corollary}[\cite {KS08}]
The knots $3_1$ and $4_1$ are minimal.
\end{corollary}

\begin{proof}
We assume that $3_1\geq K$ and $K \neq 3_1$. 
By Proposition \ref{fiber_genus}, $K$ is a fibered knot of genus one 
or the trivial knot. 
It is well known that the fibered knots of genus one are only $3_1$ and $4_1$. 
If there exists an epimorphism $G(K) \to G(K')$, 
then the Alexander polynomial of $K$ can be divided by that of $K'$ (see \cite{CF}). 
However, the Alexander polynomial of $3_1$ cannot be divided by that of $4_1$. 
Hence it follows that $3_1 \ngeq 4_1$. 
Similarly, it is shown that $4_1$ is minimal. 
\end{proof}

Furthermore, we obtain the following. 

\begin{proposition}\label{fiber_samegenus}
Let $K,K'$ be fibered knots of the same genus. 
If $K\geq K'$, then $K = K'$. 
\end{proposition}

\begin{proof}
By the proof of Proposition \ref{fiber_genus}, 
there exists an epimorphism 
between the commutator subgroups of the knot groups $G(K),G(K')$. 
By assumption, the commutator subgroups $[G(K),G(K)], [G(K'),G(K')]$  
are the free groups of the same rank. 
This implies the epimorphism 
$[G(K),G(K)]\rightarrow[G(K'),G(K')]$ is an isomorphism, 
since the free group is Hopfian. 
Then the epimorphism $G(K)\rightarrow G(K')$ is also an isomorphism. 
Therefore we obtain that $K = K'$ for prime knots $K,K'$. 
\end{proof}

\begin{corollary}
The knots $5_1$, $6_2$, and $6_3$ are minimal. 
\end{corollary}

\begin{proof}
It is known that $5_1$ is a fibered knot of genus two. 
Suppose $5_1\geq K$, then $K$ is a fibered knot of genus one or two. 
If the genus of $K$ is one, then $K$ is $3_1$ or $4_1$. 
However, the Alexander polynomial of $5_1$ can be divided 
by neither that of $3_1$ nor that of $4_1$. 
Then the genus of $K$ is two.
It implies $K = 5_1$ by Proposition \ref{fiber_samegenus}. 
By the similar argument, $6_2$ and $6_3$ are also minimal. 
\end{proof}

Boileau-Boyer-Reid-Wang \cite{BBRW} and Boileau-Boyer \cite{BB} 
studied epimorphisms of $2$-bridge knot groups. 
As an application, we get some minimal elements. 

\begin{proposition}\label{five_two}
The knots $5_2$ and $6_1$ are minimal.
\end{proposition}

\begin{proof}
By the result of Boileau-Boyer-Reid-Wang \cite{BBRW}, 
if a $2$-bridge knot group surjects onto $G(K)$, 
then $K$ is a $2$-bridge knot or the trivial knot. 
Moreover, Boileau-Boyer \cite{BB} showed that 
if there exists an epimorphism between $2$-bridge knot groups, 
namely $S(p,q) \geq S(p',q')$, then $p=kp'$ where $k>1$. 
Then $5_2=S(7,3)$ is minimal. 
Besides, it is easy to see that 
the Alexander polynomial of $6_1 = S(9,7)$ cannot be divided by that of $3_1 = S(3,1)$. 
Therefore there exists no epimorphism from $G(6_1)$ onto $G(3_1)$. 
\end{proof}

\section{Fibered knot with irreducible Alexander polynomial}

\begin{proposition}
If $K$ is a fibered knot and the Alexander polynomial $\Delta_K(t)$ is irreducible over $\Z$, 
then $K$ is minimal. 
\end{proposition}

\begin{proof}
Let $K$ be a fibered knot whose Alexander polynomial is irreducible over $\Z$. 
Suppose that $K\geq K'$ and that $K'$ is non-trivial. 
Then $\Delta_{K}(t)$ can be divided by $\Delta_{K'}(t)$ and $K'$ is fibered by Proposition \ref{fiber_genus}. 
Here the Alexander polynomial of $K'$ is not trivial. 
Since $\Delta_{K}(t)$ is irreducible, we have $\Delta_{K}(t) = \Delta_{K'}(t)$. 
It implies that the genus of $K$ is the same as that of $K'$. 
Therefore $K = K'$ by Proposition \ref{fiber_samegenus}. 
\end{proof}

\begin{example}
We consider $8_{16}$, which is a fibered knot of genus $3$. 
Since $8_{16}$ is a $3$-bridge knot, 
we cannot apply the similar argument to Proposition \ref{five_two}. 
The Alexander polynomial is $t^6 - 4 t^5 + 8 t^4 - 9 t^3 + 8 t^2 - 4 t + 1$, 
which is irreducible over $\Z$. 
Then we obtain that $8_{16}$ is minimal. 
Similarly, we can show that $8_{17}$ is also minimal. 
\end{example}

\end{document}